\documentclass[12pt,a4paper]{amsart}

\usepackage[active]{srcltx}
\usepackage {color}

\usepackage{epsfig}
\usepackage{amssymb}
\usepackage{latexsym}
\usepackage{mathtools}
\usepackage{upgreek}
\usepackage{enumerate}

\mathtoolsset{showonlyrefs} 

\newtheorem{teo}{Theorem}[section]
\newtheorem{lema}{Lemma}[section]

\newtheorem{defi}{Definition}[section]

\newtheorem{rem}{Remark}[section]

\newcommand{\om}[1][a,b]{\ensuremath{\mathcal{L}_{#1}}}
\newcommand{\km}[1][a,b]{\ensuremath{\mathcal{P}_{#1}}}
\newcommand{\kl}[1][a]{\ensuremath{P_{#1}}}
\newcommand{\kf}[1][b]{\ensuremath{P_{#1}^s}}

\newcommand{\of}[1][s]{(-\Delta)^#1}

\newcommand{\lp}[1][p]{L^#1}

\DeclareRobustCommand{\rchi}{{\mathpalette\irchi\relax}}
\newcommand{\irchi}[2]{\raisebox{\depth}{$#1\chi$}} 

\title[Fujita exponent and blow-up rate for a mixed operator]{Fujita exponent and blow-up rate for a mixed local and nonlocal heat equation}

  \author[L. M. Del Pezzo]{Leandro M. Del Pezzo}
	\address{Leandro M. Del Pezzo \hfill\break\indent
            IESTA --Facultad de Ciencias Económicas y de Administración
            \hfill\break\indent Universidad de la República
           \hfill\break\indent  Av. Gonzalo Ramírez 1926, 11200 Montevideo,
            \hfill\break\indent Departamento de Montevideo - Uruguay .}
			\email{leandro.delpezzo@fcea.edu.uy}

\author[R. Ferreira]{Ra\'ul Ferreira}
                \address{Ra\'ul Ferreira\hfill\break\indent
                Departamento de An\'alisis Matem\'atico y Matem\'atica Aplicada,
                \hfill\break\indent Fac. de C.C. Qu\'{\i}micas, U.
                Complutense de Madrid,
                \hfill\break\indent 28040,  Madrid, Spain.}
                \email{raul\_ferreira@mat.ucm.es}
\begin{document}

\begin{abstract}
	In this paper we consider the blow-up problem for a mixed
	local-nonlocal diffusion operator,
	$$
	u_t=a\Delta u -b\of u+u^p.
	$$
	We show that the Fujita exponent is given by the nonlocal part,
	$p_F=1+2s/N$. We also determinate, in some cases, the blow-up rate.
\end{abstract}

\maketitle

\section{Introduction}

	The main goal of this paper is to determine the Fujita exponent for the problem:
    \begin{equation}\label{fp}
       \begin{cases}
          u_t(x,t) - \om u(x,t) = u^p & \text{in } \mathbb{R}^N \times
          (0, T),\\
            u(x,0) = u_0(x) & \text{in } \mathbb{R}^N,
        \end{cases}
    \end{equation}
    where \( p > 1 \), \( u_0(x) \ge 0 \), \( T > 0 \) is the maximal time
    of existence, and the diffusion operator
    \( \om \) is the sum of a local operator (the Laplacian) and a
    nonlocal operator (the fractional Laplacian)
    \[
        \om u = a\Delta u - b(-\Delta)^s u.
    \]

	\medskip
	
    When \( b = 0 \), we recover the classical Fujita equation. It is
    well known, see \cite{Fujita, Weissler}, that there are two critical
    exponents: the global existence exponent \( p_0 = 1 \) and the so-called \emph{Fujita} exponent
    \[
        p_F \coloneqq 1 + \frac{2}{N},
    \]
    such that:
    \begin{enumerate}[(i)]
        \item All solutions exist globally in time if \( p \le p_0 \).
        \item All solutions blow up in finite time if
            \( p_0 < p \le p_F \).
        \item There exist both global and blow-up solutions if
            \( p > p_F \).
    \end{enumerate}

    Based on these results, extensive research has been conducted for
    different local diffusion operators, such as the $p$-Laplacian or the
    Porous Medium Equation. See, for instance, the review books
    \cite{QuittnerSouplet07, GKMS}.

    \medskip

    In the case \( a = 0 \), we obtain the semilinear fractional heat
    equation with critical exponents \( p_0 = 1 \) and
    \[
        p_F^s \coloneqq 1 + \frac{2s}{N},
    \]
    see \cite{Sugitani, NagasawaSirao, GuedaKirane}. The method of proof
    in \cite{NagasawaSirao} is probabilistic, while in \cite{Sugitani,
    GuedaKirane}, the approach is analytic. We also refer to
    \cite{MelianQuiros, Alfaro} for cases involving integrable kernels,
    where the Fujita exponent depends on the decay of the kernel at
    infinity.

    As stated earlier, our interest is in determining the Fujita exponent
    for positive parameters \( a \) and \( b \). To do this, we study two
    different problems related to \( \om u \). First, the linear problem
    \begin{equation}
        \label{cp}
        \begin{cases}
            u_t(x,t) - \om u(x,t) = 0 & \text{in } \mathbb{R}^N \times
            (0, \infty),\\
            u(x,0) = u_0(x) & \text{in } \mathbb{R}^N,
        \end{cases}
    \end{equation}
    and semilinear problem in a bounded smooth  domain (that is  the Dirichlet problem)
    \begin{equation}\label{Dirichlet}
        \begin{cases}
            u_t(x,t) - \om u(x,t) = u^p(x,t) & \text{in }
            \Omega \times (0, T),\\
            u(x,t) = 0 & \text{in } (\mathbb{R}^N \setminus \Omega)
            \times (0, T),\\
            u(x,0) = u_0(x) & \text{in } \Omega.
        \end{cases}
    \end{equation}

    The operators \( \om \) naturally emerge from the amalgamation of two
    stochastic processes operating at distinct scales: a classical random
    walk and a L\'evy flight. Essentially, when a particle can follow
    either process based on certain probabilities, the resulting limit
    diffusion equation can be characterized by an operator of the form
    \( \om \). For an exhaustive examination of this phenomenon and its
    implications, refer to the appendix in \cite{DipierroProiettiLippi}.

    The mixed operator provides a framework to analyze the disparate
    effects of “local” versus “nonlocal” diffusions in practical
    scenarios, such as examining how different forms of “regional” or
    “global” restrictions could mitigate the spread of a pandemic, as
    discussed in \cite{Epstein}. Traditional applications extend to
    domains such as heat transport in magnetized plasmas, elaborated in
    \cite{Blazevski}.

    The mathematical exploration of operators with varying orders is not
    novel. Existing literature encompasses findings pertaining to various
    scenarios; see, for instance, \cite{Jakobsen1, Jakobsen2,
    BarlesImbert, Biswas, Ciomaga, delallave, Chen1, Chen2, Mimica}.

    In the context of the Fujita problem, a significant challenge when
    dealing with the mixed operator lies in the loss of the scaling
    property. In this work, we develop new techniques to address and
    overcome this difficulty.
%
%
%
%
%

    \subsection{Main Results}

    Using the Fourier transform, we show that if
    \( u_0 \in L^1(\mathbb{R}^N) \), then there is a solution to the
    linear problem \eqref{cp}, and for large times, it behaves like the
    fractional heat equation.

    \begin{teo}\label{cpconvasym}
        Let \( u_0 \in L^1(\mathbb{R}^N) \cap L^\infty(\mathbb{R}^N) \)
        be a non-negative and non-trivial function
        and \( M = \int_{\mathbb{R}^N} u_0(x)\, dx \) be its mass. There
        exists a unique solution to \eqref{cp} given by convolution with
        the heat kernel associated with \( \om \). Moreover, it satisfies
        \[
            \lim_{t \to \infty} t^{\frac{N}{2s}} \|u(\cdot,t)
            - M\kf(\cdot,t)\|_{L^\infty(\mathbb{R}^N)} = 0,
        \]
        where \( \kf \) is the heat kernel associated with the nonlocal
        part \( (-\Delta)^s \).
    \end{teo}

    Our second aim is to study the blow-up phenomena for semilinear
    equations involving the local-nonlocal diffusion operator \( \om \),
    that is, the Dirichlet problem \eqref{Dirichlet}
    in the case that \( \Omega \subset \mathbb{R}^N \) is
    a bounded smooth  domain and the Cauchy problem
    \eqref{fp}
    in the case \( \Omega \equiv \mathbb{R}^N \).

    In the first case,
    we prove that the existence of blow-up depends on the size of the
    initial datum.

    \begin{teo}\label{teo:Dirichlet}
        Let \( \Omega \) be a
        bounded smooth  domain and \( u_0 \) be a bounded
        non-negative  continuous function. Then:
        \begin{enumerate}[i)]
            \item The solution of \eqref{Dirichlet} blows up in finite time,
            provided that
            \[
                \int_{\Omega} u_0(x)\psi_\Omega^{a,b}(x)\, dx >
                \left[ \lambda_1^{a,b}(\Omega) \right]^{\frac{1}{p-1}},
            \]
            where \( \lambda_1^{a,b}(\Omega) \) is the first Dirichlet
            eigenvalue of \( \om \), and \( \psi_\Omega^{a,b} \) is the
            corresponding eigenfunction.
            \item The solution is globally defined if
            \[
                \|u_0\|_{L^\infty(\Omega)} \ll 1.
            \]
        \end{enumerate}
    \end{teo}

    In the second case, \( \Omega = \mathbb{R}^N \), we show that the
    Fujita exponent for our mixed operator coincides with the fractional
    Fujita exponent (when the diffusion is only given by the fractional
    Laplacian), that is, \( p_F^s = 1 + \frac{2s}{N} \).

    \begin{teo}\label{teoFujita}
        Let \( u \) be a solution of \eqref{fp}.
        \begin{enumerate}[i)]
            \item If \( 1 < p \le p_F^s \) and \( u_0 \in
            L^1(\mathbb{R}^N) \cap L^\infty(\mathbb{R}^N) \) is
            non-negative and non-trivial, then \( u \)
            blows up in finite time.
            \item If \( p > p_F^s \), then there exist non-negative
            initial data \( u_0 \in L^1(\mathbb{R}^N) \cap
            L^\infty(\mathbb{R}^N) \) such that \( u \) is a global
            solution.
        \end{enumerate}
    \end{teo}

    Finally, we focus on the study of the blow-up rate of non-global
    solutions. As usual, we can find a lower bound for the solutions of
    \eqref{Dirichlet} with the flat solution \( U(t) = [(p-1)(T-t)]^{-
    \frac{1}{p-1}} \) due to the comparison principle. Moreover,
    we get
    \[
        \|u(\cdot,t)\|_{L^\infty(\Omega)} \ge c (T-t)^{-\frac{1}{p-1}}.
    \]
    We can obtain the upper bound with some restrictions.


    \begin{teo}\label{tasas}
        Let \(\Omega\) be a bounded smooth  domain,
         \( u_0 \) be a bounded
        non-negative  continuous function, and \( u \) be a solution
        of \eqref{Dirichlet} that blows up at a finite time \( t = T \).
        Then there is a positive constant \( C \) such that
        \[
            \|u(\cdot,t)\|_{L^\infty(\Omega)} \le
            C (T-t)^{-\frac{1}{p-1}}
        \]
        provided that either:

        i) the solution is strictly increasing in time, which is characterized in terms of the initial function $u_0$ as to satisfy
        $$
        \om u_0 +\mu u_0^p\ge 0 \quad\mbox{for some } 0<\mu<1.
        $$

        ii) Or,  \( 1 < p < p_F = 1 + \frac{2}{N} \) and
        \[
            B(u)\coloneqq\left\{ x\in\Omega\colon
                \begin{aligned}
                    \exists& (x_n,t_n)\in\Omega\times(0,T)
                    \text{ s.t. }\\
                    &(x_n,t_n)\to(x,T), u(x_n,t_n)\to\infty
                \end{aligned}
            \right\}      \subset\subset\Omega.
        \]
    \end{teo}

    \medskip

    While we were finalizing this manuscript, another proof of
    Theorem \ref{teoFujita} was presented in a paper \cite{Biagi2024},
    which recently appeared on arXiv.

    \medskip

    \subsection*{The article is organized as follows}
    In Section \ref{scp}, we consider the heat equation and
    the eigenvalue problem associate to the diffusion operator $\om$.
    In Section \ref{semilinear}, we point out some aspect about  the
    existence of solution to the smilinear equations and we prove a
    comparison principle. In Section \ref{sBUvsGE}, we study the blow-up
    phenomena. In Subsection \ref{sDP}
    we prove Theorem \ref{teo:Dirichlet} and in Subsection \ref{sfe}, we show
    that the Fujita exponent for our mixed operator coincides with the
    Fujita exponent for the fractional Laplacian, Theorem \ref{teoFujita}.
    Finally, in Section \ref{sBUR}, we prove  that the blow-up rates for
    $p<p_F$ are given by the ode $u'=u^p$.

\section{Two linear problems}\label{scp}
	In this section we consider two problems associate to the mixed operator
	$\om$. The eigenvalue problem and the evolution heat problem.
	\subsection{Eigenvalue Problem}
		In this subsection, we collect some results
		regarding the Diricihlet eigenvalues of
		$-\om.$

		\medskip

   		Let $\Omega$ be a bounded smooth domain. Let us define the spaces
		$$
			\widetilde {H^s}(\Omega)
			=\{u\in H^s(\mathbb R^N)\,:\, u=0 \mbox{ in }
			\mathbb R^N\setminus\Omega\},
			\text{ with } 0<s\le1
		$$
		and  the bilinear form
		\begin{align*}
		    \mathcal E_{a,b}(u,v)=& a \int_\Omega \nabla u(x,t)\nabla v(x,t) dx\\
		    & +\frac{b}2 \int_{\mathbb R^N}\int_{\mathbb R^N}
		    (u(y,t)-u(x,t))(v(y,t)-v(x,t))\frac{dxdy}{|y-x|^{N+2s}}.
		\end{align*}

		Now, we consider the linear eigenvalue problem
		\begin{equation}\label{eigenp}
		   \begin{cases}
		        -\om u=\lambda u &\text{in }\Omega,\\
		        u=0 &\text{in }\mathbb{R}^N\setminus\Omega,\\
		    \end{cases}
		\end{equation}
		depending on parameter $\lambda\in\mathbb{R}.$

		\begin{defi}
		    We say that $\lambda\in\mathbb{R}$ is an Dirichlet eigenvalue of $-\om,$
		    if \eqref{eigenp} admits a weak solution $u\in
		    \widetilde {H^1}(\Omega)\setminus\{0\},$ that is
		    \[
		        \mathcal E_{a,b}(u,v)=
		        \lambda\int_{\mathbb{R}^N}u(x)v(x) dx
		    \]
		    for all $v\in \widetilde {H^1}(\Omega).$
		    The function $u$ is a corresponding eigenfunction.
		\end{defi}

		Since $ \mathcal E_{a,b}$ is $\widetilde {H^1}(\Omega)-$coercive, there is
		in  an orthonormal $L^2(\Omega)-$basis $\{u_k\}_{k\in\mathbb{N}}$
		consisting of Dirichlet
		eigenfunctions of $-\om.$ Moreover, the corresponding eigenvalues
		$\{\lambda_k^{a,b}(\Omega)\}_{k\in\mathbb{N}}$
		can be arranged in a non-decreasing sequence
		\[
		    0<\lambda_1^{a,b}(\Omega)\le\lambda_2^{a,b}(\Omega)
		    \le\cdots\le\lambda_k^{a,b}(\Omega)\le\cdots
		\]
		with $\lambda_k^{a,b}(\Omega)\to\infty$ as $k\to \infty.$

		\medskip

		As usual, we can deduce that
		\[
		    \lambda_1^{a,b}(\Omega)=\min
		        \left\{
		            \mathcal E_{a,b}(u,u)\colon u\in \widetilde {H^1}(\Omega),
		            \|u\|_{L^2(\Omega)}=1
		        \right\}.
		\]
		About the first eigenvalue and the first eigenfunction,
		we have the following result.
		\begin{lema}
			Let $u$ be an eigenfunction of $-\om$ corresponding
			to $\lambda_1^{a,b}(\Omega)$. Then,
			\begin{enumerate}[i)]
				\item $u$ is a non-negative $C^{1,\beta}(\overline\Omega)$ function,
					for some there $\beta\in(0,1)$.
				\item  $\lambda_1^{a,b}(\Omega)$ is simple.
			\end{enumerate}
		\end{lema}
		\begin{proof}
			That the function $u$ is non-negative follows from the fact that for any
			$v\in \widetilde {H^1}(\Omega)$ we have that
			\[
				\mathcal E_{a,b}(|v|,|v|)\le \mathcal E_{a,b}(v,v),
			\]
			and
			\[
				\mathcal E_{a,b}(|v|,|v|)< \mathcal E_{a,b}(v,v)
			\]
			if both \(\{v<0\}\) and  \(\{v>0\}\)  have  positive measure.
			The regularity follows by \cite[Theorems 2.7 and 3.5]{Biagi2021}.

			Finally the proof that $\lambda_1^{a,b}(\Omega)$ is simple is inspired in an
			argument of \cite{ServaEigen}.
			Suppose that $u,$ $v$ are two eigenfunctions of $-\om$ corresponding to
			$\lambda_1^{a,b}(\Omega),$ such that $u\neq v.$ Without loss of generality,
		    we can assume that $u,v$ are non-negative and
		    $\|u\|_{\lp[2](\Omega)}=\|v\|_{\lp[2](\Omega)}=1.$

		    Since $u,$ $v$ are two eigenfunctions of $-\om$ corresponding to
		    $\lambda_1^{a,b}(\Omega),$ so is $w=u-v.$ Then $\{w<0\}$ or $\{w>0\}$
		    has zero measure. Then, without loss of generality, we can assume that
		    $w\ge 0$ a.e. in $\Omega.$ Therefore $u\ge v$ a.e. in $\Omega.$

		    Thus, $\|u\|_{\lp[2](\Omega)}=\|v\|_{\lp[2](\Omega)}=1$ and
		    $u\ge v$ a.e. in $\Omega,$ we have
		    \[
		        0\le\int_\Omega (u-v)^2 dx=2-2\int_\Omega u(x)v(x) dx\le 0.
		    \]
		    Then $u=v$ a.e. in $\Omega.$
		\end{proof}
%
%
%
%

	    On the other hand, it is easy to see that
    	\begin{equation}\label{cota-inf-autovalor}
    	    \lambda_1^{a,b}(\Omega)
    	    \ge \max\{a\sigma_1(\Omega), b \mu_1^s(\Omega)\}
    	\end{equation}
    	where $\sigma_1(\Omega)$ and $\mu_1^s(\Omega)$ are the first Dirichlet eigenvalue
    	of $-\Delta$ and $(-\Delta)^s$ respectively.
    	that is
    	\[
    		\begin{array}{rl}
        		\sigma_{1}(\Omega)= & \min
        	    \left\{
        	        \|\nabla u\|_{\lp[2](\Omega)}^2\colon
        	        u\in H^1_0(\Omega), \|u\|_{L^2(\Omega)}=1
        	    \right\},\\
        		\mu_1^s(\Omega)=&\min\left\{[u]_s^2
        	    \colon u\in \widetilde{H^s}(\Omega) \text{ and }
        	    \|u\|_{\lp[2](\mathbb{R}^N)}=1
        	    \right\},
    		\end{array}
    	\]
    	where
    	\[
    		[u]_s\coloneqq \left[\int_{\mathbb R^N}\int_{\mathbb R^N}
		    \frac{|u(y)-u(x)|^2}{|y-x|^{N+2s}} dxdy\right]^\frac12
    	\]
    	is the Galiardo seminorm in \(\widetilde{H^s}(\Omega).\)
    	Moreover, it is not difficult to show the  following result.
    	
    	\begin{lema}\label{lema:ba0}
        	Let $\psi^{a,b}$ be the positive eigenfunction of $-\om$ associated to
        	$\lambda_1^{a,b}(\Omega)$ normalized such that
        	$\|\psi^{a,b}\|_{\lp[2](\Omega)}=1$. Then,
        	\begin{enumerate}[i)]
				\item \(\lambda_{a,b}(\Omega)\to b \mu_{1}^s(\Omega)\) and
					\(\psi^{a,b} \to \phi\) strongly in \(\widetilde{H^s}(\Omega)\)
					as \(a\to 0^+,\) where $\phi$ is the positive eigenfunction
					of $(-\Delta)^s$ associated to
					$\mu_1^s(\Omega)$ normalized such that
					$\|\phi\|_{\lp[2](\Omega)}=1.$
				\item \(\lambda_{a,b}(\Omega)\to b \sigma_{1}(\Omega)\) and
					\(\psi^{a,b} \to \varphi\) strongly in \({H^1_0}(\Omega)\)
					as \(b\to 0^+,\) where $ \varphi$ is the positive eigenfunction
					of $-\Delta$ associated to
					$\sigma_1(\Omega)$ normalized such that
					$\| \varphi\|_{\lp[2](\Omega)}=1.$
			\end{enumerate}
  		\end{lema}
		\begin{proof}
			We only consider the first case, the second one is similar. Notice that
			$\lambda^{a,b}_1(\Omega)$ is a decreasing function in $a$ and by
			\eqref{cota-inf-autovalor} it is bounded from bellow,
	   		\[
			    \lim_{a\to0^+}\lambda_1^{a,b}(\Omega)\ge b\mu_1^s(\Omega).
			\]

			On the other hand, given $g\in C_c^\infty(\Omega)\setminus\{0\}$  we have that
			\[
				\frac{b[g]_s^2}{\|g\|_{\lp[2](\mathbb{R}^N)}}=\lim_{a\to0^+}
			      \frac{\mathcal E_{a,b} (g,g)}{\|g\|_{\lp[2](\mathbb{R}^N)}}
			      \ge  \lim_{a\to0^+} \lambda_1^{a,b}(\Omega).
			\]
			Now, using that $ C_c^\infty(\Omega)$ is dense in
			$\widetilde{H^s}(\Omega)$ (see for instance \cite[Theorem 1.4.22]{Grisvard}),
			we get
			\[
		    	\lim_{a\to0^+}\lambda_1^{a,b}(\Omega)=b\mu_1^s(\Omega).
			\]
			Finally, it is easy to see that $\psi^{a,b}\to\phi$ strongly in
			$\widetilde{H^s}(\Omega)$
			as $a\to0^+$.
		\end{proof}

		To end this subsection we consider the particular case $\Omega=B_R(0)$.

    	\begin{lema}\label{ResAut1}
    	    If $\psi^{a,b}$ is a non-negative eigenfunction of  of $-\om$ corresponding
    	    to $\lambda_1^{a,b}(B_R(0)),$ then $\psi^{a,b}$ is spherically symmetric and
    	    radially decreasing in $B_R(0).$
    	\end{lema}
    	\begin{proof}
			Let $\psi^{a,b}_*$ the symmetric decreasing rearrangements of the function
			positive eigenfunction $\psi^{a,b}$.
			Since $\|\psi^{a,b}_*\|_{L^2(B_R(0))} =\|\psi^{a,b}\|_{L^2(B_R(0))}$ and by the
			local and non-local P\'olya-Szeg\H{o} inequalities (see for instance
			\cite{Almgren,Polya}), gives us
			$$
				\mathcal E_{a,b} (\psi^{a,b}_*,\psi^{a,b}_*)\le
				\mathcal E_{a,b} (\psi^{a,b},\psi^{a,b}).
			$$
			Therefore, the simplicity of $\lambda_1^{a,b}(B_R(0))$ implies
			$\psi^{a,b}_*=\psi^{a,b}$ and the result follows.
		\end{proof}

    	\begin{rem}\label{remark:rescalesauto}
			Let $R>0$ and $a_R= a R^{-2(1-s)}.$
			By scaling we note that, if $\psi^{a,b}_R$ is the
			non-negative eigenfunction of $-\om[a,b]$
			corresponding to $\lambda_1^{a,b}(B_R(0))$ with
			$\|\psi^{a,b}_R\|_{\lp[1](\mathbb{R}^N)}=1,$
			then
			\[
				\lambda_1^{a,b}(B_R(0))= R^{-2s}\lambda_1^{a_R,b}(B_1(0))
			 	\text{ and }\psi^{a,b}_R(x)=R^{-N}\psi^{a_R,b}_1(\tfrac{x}{R}).
			\]
   \end{rem}

	\subsection{Linear evolution problem}
		In this subsection, we consider the Cauchy
		problem for the linear heat equation \eqref{cp} and
		we prove Theorem \ref{cpconvasym}.
		
		\begin{proof}[Proof of Theorem \ref{cpconvasym}]

    	As usual, we assume $u_0\in\lp[1](\mathbb{R}^N).$ Taking the Fourier transform
    	in the spatial variable, we find
		\begin{equation}\label{cpf}
		   	 \begin{cases}
		        \hat{u}_t(\upxi,t)+(a|\upxi|^2+b|\upxi|^{2s})\hat{u}(\upxi,t) =0
		        &\text{ in } \mathbb{R}^N\times(0,\infty),\\
		        \hat{u}(\upxi,0)=\hat{u}_0(\upxi)&\text{ in } \mathbb{R}^N.
		    \end{cases}
    	\end{equation}
		Then
		\[
		    \hat{u}(\upxi,t)=\hat{u}_0(\upxi) e^{-(a|\upxi|^2+b|\upxi|^{2s})t},
		\]
		therefore the solution of \eqref{cp} is
		\[
		    u(x,t)=
		    \int_{\mathbb{R}^N}\km(x-y,t)u_0(y) dy
		\]
		where
		\[
		        \km(x,t)\coloneqq
		        \int_{\mathbb{R}^N}\kl(x-y,t)\kf(y,t) dy.
		\]
		Here $\kl$ and $\kf$ denote the heat kernel and the fractional heat
		kernel respectively, that is
		\[
		    \kl(x,t)\coloneqq \frac{e^{\frac{-|x|^2}{4at}}}{(4\pi at)^{\frac{N}2}}
		    \quad\text{ and }\quad
		    \kf(x,t)\coloneqq
		    \frac1{(2\pi b t^{\frac1{s}})^{\frac{N}2}}\int_{\mathbb{R}^N}
		    e^{i\frac{x}{bt^{1/2s}}\cdot \upxi-| \upxi|^{2s}}
		    d\upxi.
		\]
		Although the function $\kf$ does not have an explicit formula,
		it is well known that it has a self-similar form
		$$
			\kf(x,t)= t^{\frac{-N}{2s}} F(x t^{\frac{-1}{2s}}),
		$$
		where the profile $F$ is a regular bounded function,
		see for instance \cite{Garafolo}. Therefore,
%
%
$u\in C^\infty(\mathbb R^N\times (0,\infty))$ and satisfies the following decay estimate,
    \begin{equation}\label{eq:cauchycota}
        \|u(\cdot,t)\|_{L^\infty(\mathbb{R}^N)}\le\frac{C}{t^{\frac{N}{2s}}}
        \|u_0\|_{L^1(\mathbb{R}^N)}.
    \end{equation}

    Now, in order to study the asymptotic behaviour of $u$, we know that there is a positive constant $C$ depending only on $N$ such
        that
        \begin{equation}\label{eqNorinf}
            \begin{aligned}
                \|u(\cdot,t)-M\kf(\cdot,t)\|_{\lp[\infty](\mathbb{R}^N)}
                &\le C\|\hat{u}(\cdot,t)-M\hat{\kf}(\cdot,t)\|_{\lp[1](\mathbb{R}^N)}\\
                &\le C\int_{\mathbb{R}^N}e^{-b|\upxi|^{2s}t}
                \left|e^{-a|\upxi|^{2}t}\hat{u}_0(\upxi)-M\right|d\upxi\\
                &= C t^{-\frac{N}{2s}} I(t),
            \end{aligned}
        \end{equation}
        where
        $$
        I(t)= \int_{\mathbb{R}^N}e^{-b|z|^{2s}}
                \left|e^{-a|z|^{2}t^{1-\frac1{s}}}\hat{u}_0(zt^{-\frac1{2s}})-M\right|dz
        $$
        Since $\hat{u}$ is uniformly continuous,
        $\|\hat{u}_0\|_{\lp[\infty](\mathbb{R}^N)}\le  \|u_0\|_{L^1(\mathbb{R}^N)}  $, $M=\hat{u}_0(0)$ and
        $e^{-b|z|^{2s}}\in\lp[1](\mathbb{R}^N)$ we can apply dominated convergence theorem to get
        \begin{equation}\label{eqI3}
            \lim_{t\to\infty}I(t)=0.
        \end{equation}
	\end{proof}
\section{Semilinear problem}\label{semilinear}

	In this section, we highlight certain aspects concerning the existence
	and comparison of the semilinear problem in
	 two scenarios: when $\Omega
	\subset \mathbb{R}^N$ is a bounded smooth domain
	(the Dirichlet problem \eqref{Dirichlet}),
	and when $\Omega \equiv \mathbb{R}^N$
	(the Cauchy problem \eqref{fp})

	\medskip	

	We start with the case $\Omega$ is a  bounded  smooth domain. We consider
	weak solution  the problem of \eqref{Dirichlet}.

	\begin{defi}
		We say that a function
		$u\in L^2((0,T);\widetilde {H^1}(\Omega))$
		with $u_t\in L^2((0,T),H^{-1}(\Omega))$ is
		a weak supersolution (subsolution) of \eqref{Dirichlet} if
		\[
			\int_\Omega u_t \varphi dx \ge (\le) - \mathcal E_{a,b}
			(u,\varphi) +\int_\Omega u^p \varphi dx
			\qquad\forall \varphi\in L^2((0,T);\widetilde {H^1}(\Omega)),
		\]
		and $u(x,0)\ge(\le)u_0(x)$.
		We say that $u$ is a solution
		if it is both a supersolution and a subsolution.
	\end{defi}
	We note that, in view of the regularity of $\partial\Omega$,
	it is well-known that $\widetilde{H^1}(\Omega)$
	can be naturally identified with $H_0^1(\Omega)$ and
	$$
		a\|u\|_{H_0^1(\Omega)}
			\le \mathcal E_{a,b}(u,u) \le C \|u\|_{H_0^1(\Omega)}.
	$$
	Therefore, we can apply, step by step,
	the same argument as in the classical semilinear heat equation,
	to get the existence a uniqueness of weak solution.
	We also have a comparison principle.

	\begin{teo}\label{weak-comparison}
		Let $\overline u$ and $\underline u$ be a bounded weak supersolution
		and a bounded weak subsolution of \eqref{Dirichlet} respectively.
		Then, $\underline u\le \overline u$.
	\end{teo}
	\begin{proof}
		Let $w=\underline u- \overline u$.
		Taking as test function $\psi=w_+$ we have that
		$$
			\int_{\mathbb R^N}w_t w_+ dx  \le \mathcal E(w,w_+)
			+\int_{\mathbb R^N} (\underline u^p-\overline u^p) w_+ dx.
		$$
		Notice that, $\mathcal E(w,w_+)\ge \mathcal E(w_+,w_+)\ge 0, $ and
		$$
			\int_{\mathbb R^N}w_t w_+ dx
			=\frac12 \partial_t \int_{\mathbb R^N} (w_+)^2 dx,
		$$
		and
		$$
			\int_{\mathbb R^N} (\underline u^p-\overline u^p)
			w_+ dx\le  C \int_{\mathbb R^N} (w_+)^2 dx.
		$$
		Hence,
		$$
 			\partial_t \int_{\mathbb R^N} (w_+)^2 dx
 			\le C \int_{\mathbb R^N} (w_+)^2 dx,
 			\qquad w_+(x,0)=0
		$$
		and the result follows.
	\end{proof}

	In the case $\Omega=\mathbb R^N$, we use the concept of mail solution.

	\begin{defi} Let $X\coloneqq L^1(\mathbb R^N)\cap L^\infty(\mathbb R^N)$.
	  Given $u_0\in X$ a function $u\in C([0,T);X)$
	  is said to be a mild solution of \eqref{fp} if
		\[
			u(x,t)
			=\int_{\mathbb R^N}\km(x-y,t) u_0(y) dy
			+\int_0^t \int_{\mathbb R^N} \km(x-y,t-s)u^p(y,s) dy ds.
		\]
		a.e $x\in\mathbb{R}^N$ and for all $0<t<T$.
	\end{defi}

	\begin{teo}\label{teo-exist}
		For each $u_0\in X$ there exists a unique mild solution
		$u\in C((0,T);X)$ of \eqref{fp}.
		Moreover, $u\in C^\infty(\mathbb R^N\times (0,T))$.
	\end{teo}
	\begin{proof}
		Let $\|u_0\|_X=\|u_0\|_{L^1(\mathbb R^n)}
			+\|u_0\|_{L^\infty(\mathbb R^n)}=M$ and,
		for $T_0>0$ fixed, consider the space
		$$
			E\coloneqq\{v\in C((0,T_0);X)
			\colon\sup_{0<t<T_0}\|v(\cdot,t)\|_X\le 4M\}.
		$$
		We define the operator
		\begin{equation}\label{operator-phi}
			\Phi(v)(t)=\int_{\mathbb R^N}\km(x-y,t) u_0(y) dy
			+\int_0^t \int_{\mathbb R^N} \km(x-y,t-s)v^p(y,s) dy ds.
		\end{equation}
		We want to prove that if $T_0$ is small then $\Phi:E\to E$ is
		contractive, and thus has a unique fixed point. Let us observe that
		\begin{align*}
			\|\Phi(v)(t)\|_{L^\infty(\mathbb R^n)}
			&\le \|u_0\|_{L^\infty(\mathbb R^n)}+\int_0^t
			\|v(\tau)\|_{L^\infty(\mathbb R^n)}^p\,d\tau\\
			&\le M+cM^p T_0\le2M,
		\end{align*}
		and also
		\begin{align*}
			\|\Phi(v)(t)\|_{L^1(\mathbb R^n)}&\le \|u_0\|_{L^1(\mathbb R^n)}
			+\int_0^t
			\|v(\tau)\|_{L^\infty(\mathbb R^n)}^{p-1}
			\|v(\tau)\|_{L^1(\mathbb R^n)}\,d\tau\\
			&\le M+cM^p T_0\le2M,
		\end{align*}
		provided $T_0$ is small enough.
		Thus $\Phi(E)\subset E$. Similarly, for $v_1,v_2\in E$,
		$$
			\begin{array}{rl}
				\displaystyle\|\Phi(v_1)(t)-\Phi(v_2)(t)\|_X&\displaystyle\le 		
				\int_0^t (4M)^{p-1}\|v_1(\tau)-v_2(\tau)\|_{X}\,d\tau
				\\ [4mm]
				&\displaystyle\le c T_0 M^{p-1}\sup\limits_{0<\tau<T_0}
				\|v_1(\tau)-v_2(\tau)\|_X\\ [3mm]
				&\displaystyle\le\frac12 \|v_1-v_2\|_E,
			\end{array}
		$$
		and $\Phi$ is contractive in $E$. Finally, the regularity of the
		function $u$ is given by the regularity of the kernel $\km$.
%
	\end{proof}

	\begin{defi}
		Let   $u\in C^{2,1}_{x,t}(\mathbb R^N\times (0,T))$
		a function with an admissible growth at infinity
		$$
		\int_{\mathbb R^N} \frac{|u(x)|}{(1+|x|)^{N+2s}} dx<\infty.
		$$
		We say that it is a classical supersolution (subsolution) if
		$$
		\left\{
		\begin{array}{rl}
		u_t \ge (\le) \om u+u^p & x\in\mathbb R^N, t\in(0,T)\\
		u(x,0)\ge(\le) u_0(x) &x\in\mathbb R^N.
		\end{array}\right.
		$$
	\end{defi}

	\begin{teo}
		Let $\overline u$ and $\underline u$ be a bounded classical
		supersolution and a bounded classical subsolution
		of \eqref{fp} respectively. Then,
		$\underline u\le \overline u$.
	\end{teo}
	\begin{proof}
		First we note that $\phi(x)=(1+|x|^2)^{\gamma/2}$
		with $\gamma<2s$ is a regular function which satisfies
		$$
			\lim_{|x|\to\infty}\phi(x)=\infty,\qquad \om \phi (x)\le K.
		$$
		Now, let us define $w=\underline u-\overline u - \varepsilon e^{\lambda t} \phi(x)$. Since $\overline u$ and $\underline u$ are bounded, we have
		$$
		w(x,t)<0 \mbox{ in } B_R^c(0)\times(0,T).
		$$
		Moreover,
		$$
			\begin{array}{rl}
				w_t-\om w\le & \underline u^p-\overline u^p -
				 \varepsilon\lambda e^{\lambda t} \phi(x)+
				\varepsilon e^{\lambda t} \om \phi \\
				 =& C|\xi|^{p-1}(\underline u-\overline u) -
				 \varepsilon\lambda e^{\lambda t} \phi(x)+
				\varepsilon e^{\lambda t} K\\
				\le&
				C w  + \varepsilon e^{\lambda t} \left(\phi(x)
				-\lambda \phi(x)+K\right).
			\end{array}
		$$
		Assume that there exists a first time $t_0$ such that $w(x_0,t_0)=0$
		for some $x_0\in B_R(0)$. Then, for $\lambda$ large enough, the
		previous inequality reads
		$$
			w_t-\om w\le \varepsilon e^{\lambda t_0}
			\left(\phi(x_0)-\lambda \phi(x_0)+K\right) <0
		$$
		which contradicts the fact that
		$w_t(x_0,t_0)\ge 0$ and $\om w(x_0,t_0)<0$.
	\end{proof}

\section{Blow-up versus global existence} \label{sBUvsGE}

	Here, we consider the semilinear problems  and prove
	Theorems \ref{teo:Dirichlet} and \ref{teoFujita}.

	\subsection{Dirichlet Problem} \label{sDP}

		In this section we consider the problem \eqref{Dirichlet}.
%

\begin{proof}[Proof of Theorem \ref{teo:Dirichlet}]
{\it i)} To obtain blow-up condition we follow the ideas of \cite{Kaplan}.
		Let	\(u\) be a weak solution of \eqref{Dirichlet} and
		$\psi_\Omega^{a,b}$ the non-negative eigenfunction of
		$-\om$ associated to the first Dirichlet eigenvalue
		$\lambda_1^{a,b}(\Omega)$ normalized such that
		$\|\psi_\Omega^{a,b}\|_{L^1(\mathbb{R}^N)}=1.$

		Since both functions $u(\cdot,t)$ and $\psi_\Omega^{a,b}$ are in
		the space $\widetilde{H^1}(\Omega)$ we can use one of them
		as a test function in the other's problem to obtain
		that the function
        \[
            J(t)=\int_{\Omega} u(x,t)\psi_\Omega^{a,b}(x) dx
        \]
        satisfies
         \begin{align*}
                        J_t(t)=&\int_{\Omega} u_t(x,t)\psi_\Omega^{a,b}(x) dx =
                        -  \mathcal E (u,\psi_\Omega^{a,b}) +
                        \int_{\Omega}
                        u^p(x,t)\psi_\Omega^{a,b}(x) dx\\
                        =& -\lambda_1^{a,b}(\Omega) J(t) + \int_{\Omega}
                        u^p(x,t)\psi_\Omega^{a,b}(x) dx.
         \end{align*}

        Then, by Jensen's inequality,
        \[
            J_t(t)\ge J^p(t)-\lambda_1^{a,b}(\Omega) J(t).
        \]
        Therefore, if
        \begin{equation}\label{fujitaimportante1}
                  \int_{\Omega} u_0(x)\psi_\Omega^{a,b}(x) dx
                  >\left[\lambda_1^{a,b}(\Omega)\right]^{\frac1{p-1}}
        \end{equation}
        then $J$ (and thus $u$) blows up in finite time.

      {\it ii)}
      We only note that  $v=A \psi_{\Omega}^{a,b}(x)$
      is a supersolution of
      \eqref{Dirichlet} if
      \[
      		A\psi_{\Omega}^{a,b}(x)\le
      		\left[\lambda_1^{a,b}(\Omega)\right]^{1/(p-1)}
      		\text{ in } \Omega,
      \]
      then for small initial data the solution are global.

\end{proof}
\subsection{Cauchy Problem. Fujita exponent}\label{sfe}
    In this section, we prove that the Fujita exponent for our mixed operator coincides
    with the fractional Fujita exponent.

    \begin{proof}[Proof of Theorem \ref{teoFujita}]
        {\it i)} We split the proof
        in two cases.

        \medskip

        \noindent{\it Case 1: $1<p<p_F^s.$} Notice that the solution of the Dirichlet
        problem \eqref{Dirichlet} is a subsolution of the Cauchy problem \eqref{cp}.
        Thus, by comparison principle, if $v$ blows up then $u$ also blows up.

        Taking $\Omega=B_R(0)$, we can rewrite the blow-up condition
        \eqref{fujitaimportante1} as
        \begin{equation}\label{fujitaimportante2}
            \int_{B_R(0)} v_0(x)\psi_1^{a_R,b}\left(\dfrac{x}R\right) dx>
            R^{N-\frac{2s}{p-1}}
            \left[\lambda_1^{a_R,b}(B_1(0))\right]^{\frac1{p-1}}
        \end{equation}
        where $a_R=R^{2(s-1)}a$, see Remark \ref{remark:rescalesauto}.

        By Lemma \ref{lema:ba0}, if $p<p_F^s$ then
        \[
            \lim_{R\to\infty}R^{N-\frac{2s}{p-1}} \left[\lambda_1^{a_R,b}(B_1(0))\right]^{\frac1{p-1}}=0.
        \]

        On the other hand, using that $\psi_1^{a,b}$ is a radially decreasing function (see Lemma \ref{ResAut1}), given $R_0<R$ we get
        \[
            \int_{B_R(0)} v_0(x)\psi_1^{a_R,b}\left(\dfrac{x}R\right) dx>
            \int_{B_{R_0}(0)} v_0(x)\psi_1^{a_R,b}\left(\dfrac{x}R_0\right) dx.
        \]
        Again, by Lemma \ref{lema:ba0}, we have
        \[
             \lim_{R\to\infty}\int_{B_R(0)} u_0(x)\psi_1^{a_R,b}\left(\dfrac{x}R\right) dx\ge
              \int_{B_{R_0}(0)} u_0(x)\psi_s\left(\dfrac{x}R_0\right) dx,
        \]
        where $\psi_s$ is the non-negative eigenfunction of $\of$ associated to
        $\mu_1(B_1(0))$ with $\|\psi_s\|_{\lp[1](\mathbb{R}^N)}=1.$ Letting
        $R_0\to \infty,$ we get
        \[
             \lim_{R\to\infty}\int_{B_R(0)} u_0(x)\psi_1^{a_R,b}\left(\dfrac{x}R\right) dx\ge
              \psi_s\left(0\right)\int_{\mathbb R^N} u_0(x) dx>0.
        \]
        Therefore \eqref{fujitaimportante1} holds for large $R.$

         \medskip

        \noindent{\it Case 2: $p=p_F^s.$}
        The proof is a modification of the arguments in \cite{MelianQuiros,Mitidieri}.

        By contradiction, suppose that there is
        a non-trivial and non-negative global solution $u$ of \eqref{fp}. Then, $u$ does not verify the blow-up condition
        \eqref{fujitaimportante2}. That is, taking $0<R_0<R$ we have
        \begin{align*}
                \int_{B_{R_0}(0)} u(x,t)\psi_1^{a_R,b}\left(\dfrac{x}{R_0}\right) dx
                  &\le \int_{B_R(0)} u(x,t)\psi_1^{a_R,b}\left(\dfrac{x}R\right) dx\\
                  &\le
                               \left[\lambda_1^{a_R,b}(B_1(0))\right]^{\frac1{p-1}}
                \quad\forall t>0.
        \end{align*}

        Now, using Lemma \ref{lema:ba0}, and passing to the limit first $R\to\infty$
        and second $R_0\to\infty,$ we get
        \begin{equation}\label{cota L1}
                \int_{\mathbb{R}^N} u(x,t) dx
                  \le
               \frac{\left[\mu_1^s(B_1(0))\right]^{\frac1{p-1}}}{\psi_s(0)}
               \quad\forall t>0,
        \end{equation}
        where $\mu_1^s(B_1(0))$ is the first Dirichlet eigenvalue of $(-\Delta)^s.$
        Then the set
        $\{u(\cdot,t)\}_{t>0}$ is bounded in $L^1(\mathbb{R}^N).$
        On the other hand, integrating the equation in \eqref{cp} in $\mathbb R^N$ we get
        $$
        \partial_t \int_{\mathbb R^N} u(x,t)dx =\int_{\mathbb R^N} u^p(x,t) dx.
        $$
        Integrating in $(0,t)$ we have
        $$
        \int_{\mathbb R^N} u(x,t)dx -\int_{\mathbb R^N} u_0(x)dx=\int_0^t \int_{\mathbb R^N} u^p(x,t)dx.
        $$
        Thanks to \eqref{cota L1}, we can take $t\to \infty$ to obtain that
        \begin{equation}\label{fujitaimportante3}
            \int_0^\infty \int_{\mathbb R^N} u^p(x,t)dx\le C.
        \end{equation}

        We know chose $\xi\in C_0^\infty(B_1(0))$ and $\eta\in C^\infty(-1,1)$ such that
        $0\le\xi,\eta\le1,$ $\xi=1$ in $B_{\tfrac12}(0),$ and
        $\eta\equiv1$ in $[0,\tfrac12).$ For fixed $t_0>0,$
        we set $\xi_R(x)=\xi(\tfrac{x}R)$ and $\eta_R(t)=\eta(\frac{t-t_0}{R^{2s}}).$
        Then
        \begin{align*}
             \int_{t_0}^\infty \int_{\mathbb{R}^N}
                 u_t(x,t)\xi_R(x)\eta_R(t) dx dt
                 =& \int_{t_0}^\infty \int_{\mathbb{R}^N}
                     \om u(x,t)\xi_R(x)\eta_R(t)dx dt\\
                  &\quad
                      +\int_{t_0}^\infty \int_{\mathbb{R}^N}
                       u^p(x,t)\xi_R(x)\eta_R(t)dx dt.
        \end{align*}
        Using that
        \[
            |\eta_R^\prime(t)|\le \dfrac1{R^{2s}}\rchi_{[t_0+\frac{R^2s}2,t_0+R^{2s}]}(t)
            \quad \forall t\ge t_0
        \]
         we get
        \begin{align*}
           &\mathcal{I}_R(t_0)\coloneqq \int_{t_0}^\infty \int_{\mathbb{R}^N}
                       u^p(x,t)\xi_R(x)\eta_R(t)dx dt\\
                &\le
                 -\int_{t_0}^\infty \int_{\mathbb{R}^N}
                 u(x,t)\xi_R(x)\eta_R^\prime (t) dx dt\\
               &\qquad  +\int_{t_0}^\infty \int_{\mathbb{R}^N}
                      u(x,t) \om\xi_R(x)
                      \eta_R(t)dx dt\\
                  &\le \frac{C}{R^{2s}}\!\! \int_{t_0+\tfrac{R^{2s}}2}^{t_0+R^{2s}}
                  \int_{B_R(0)}\!\!\!\!\!\!\!\!u(x,t)dx dt+\frac{a}{R^2}
                     \int_{t_0}^{t_0+R^{2s}}\!\!\!\! \int_{A_{R}} \!\!\!\!
                 u(x,t)|\Delta \xi(\tfrac{x}R)| dx dt \\&
                      \qquad+\frac{b}{R^{2s}}\int_{t_0}^{t_0+R^{2s}}
                      \int_{D_R}
                 u(x,t)|(-\Delta)^s \xi(\tfrac{x}R)| dx dt,
        \end{align*}
        where $A_R\coloneqq B_R(0)\setminus B_{\frac{R}2}(0)$ and
        $D_R\coloneqq \mathbb{R}^N\setminus B_{\frac{R}2}(0)$. Since the function
        $\xi\in C_0^\infty(B_1(0))$ we have
        \begin{align*}
           & \mathcal{I}_R(t_0)
                \le \frac{C}{R^{2s}}
                 \left[
                 \int_{t_0+\tfrac{R^{2s}}2}^{t_0+R^{2s}}
                  \int_{B_R(0)}\!\!\!\!\!\!\!\!u(x,t)dx dt\right.\\
                  &\qquad+\left.\frac{1}{R^{2(1-s)}}
                  \int_{t_0}^{t_0+R^{2s}}\!\!\!\! \int_{A_{R}} \!\!\!\! u(x,t) dx dt
                  +\int_{t_0}^{t_0+R^{2s}}\!\!\!\!\! \int_{D_R}
                 u(x,t) dx dt \right].
        \end{align*}
       Moreover, using H\"older inequality and $p=p_F^s$, we get
      \begin{align*}
            \mathcal{I}_R(t_0)
                &\le C  \left[
                 \left(\int_{t_0+\tfrac{R^{2s}}2}^{t_0+R^{2s}}
                  \int_{B_R(0)}\!\!\!\!\!\!\!\!u^p(x,t)dx dt\right)^{\tfrac1p}\right.\\
                  &\qquad+\frac{1}{R^{2(1-s)}}
                  \left(\int_{t_0}^{t_0+R^{2s}}\!\!\!\! \int_{A_{R}} \!\!\!\! u^p(x,t) dx dt
                  \right )^{\tfrac1p}\\
                  &\qquad \left.+\left(\int_{t_0}^{t_0+R^{2s}}\!\!\!\!\!
                  \int_{D_R}
                 u^p(x,t) dx dt \right )^{\tfrac1p} \right].
        \end{align*}
        Finally, using \eqref{fujitaimportante3} and passing to the limit as $R\to\infty$
        in the last inequality, we get
        \[
            \int_{t_0}^\infty\int_{\mathbb{R}^N} u^p(x,t) dt =0
        \]
        which is a contradiction.

       \medskip

       {\noindent {\it ii)}} As usual, by comparison principle, it is
       enough to show that there is a non-negative global super-solution of \eqref{fp}
       when $p>p_F^s.$

       Let $v_0\in \lp[1](\mathbb{R}^N)\cap \lp[\infty](\mathbb{R}^N)$ non-negative.
       We take, $v$ be the solution of the linear problem \eqref{cp}
       with initial data $ v_0,$
       \[
            w(x)=h(t) v(x,t+t_0)
       \]
       where $t_0>0$ and $h$ will be chosen later.
       Then $w$ is a super-solution of \eqref{fp} with initial data
       $u_0(x)=v(x,t_0)$
       if $h(0)=1,$ and
       \[
            \frac{h^\prime(t)}{h^{p}(t)}\ge v^{p-1}(x,t)
       \]
       for any $(x,t)\in\mathbb{R}^N\times(0,\infty).$ By
       \eqref{eq:cauchycota}, it is
       enough to take
       \[
               \frac{h^\prime(t)}{h^{p}(t)}=
               \left( C\|v_0\|_{\lp[1](\mathbb{R}^N)}
               (t+t_0)^{-\tfrac{N}{2s}}\right)^{p-1}
       \]
       Therefore
       \[
           h(t)=\left[
                    1-D t_0^{1-\frac{N}{2s}(p-1)}
                    \left(1-\left(\frac{t}{t_0}
                    +1\right)^{1-\frac{N}{2s}(p-1)}\right)
                   \right]^{\frac1{p-1}}
       \]
       where
       $D=\tfrac{2s(p-1)}{N(p-1)-2s}
       \left( C\|v_0\|_{\lp[1](\mathbb{R}^N)}\right)^{p-1}.$

       Since $p>p_F^s,$ we get
       \[
           1-\frac{N}{2s}(p-1)<0.
        \]
        Thus, taking $t_0>0$ large enough such that
        \(D t_0^{1-\frac{N}{2s}(p-1)}<1,\)
        we have that $h(t)$ is global and thus
        so is $w.$
    \end{proof}

\section{Blow-up rates}\label{sBUR}
    We now show that under certain restriction the blow-up rates are given by the natural one, that is, the rate is given by the ODE
    $$U_t=U^p.$$

    \begin{proof}[Proof of Theorem \ref{tasas}]

    Since $\om u^p\ge pu^{p-1} \om u$, the first part (increasing in time solutions) of the proof is analogous to the proof of Theorem 1.2 in \cite{RaulArturo},  so we have decided to omit it.

    For the second part we also follows the ideas given in \cite{RaulArturo}. The principal difficulty in this case is due to the fact that the mixed operator does not scale.

        Starting by defining
        \[
            M(t)\coloneqq \max\{u(x,\tau)\colon \tau\in(0,t)\}\quad\forall t\in(0,T),
        \]
        and
        \begin{align*}
            &t_0\in(0,T) \text{ (arbitrary), and } \\
            &t_n=\sup\{t\in(t_{n-1},T)\colon M(T)=2M(t_{n-1})\} \quad\forall n\in \mathbb{N}.
        \end{align*}
        Then
        \begin{itemize}
            \item $M$ is a non-decreasing function;
            \item $\{t_n\}_{n\in\mathbb{N}}$ is an increasing sequence;
            \item $\|u(\cdot,t)\|_{\lp[\infty](\Omega)}\le M(t)$ and
                $\|u(\cdot,t_n)\|_{\lp[\infty](\Omega)}= M(t_n)$
                for all $n\in\mathbb{N}.$
        \end{itemize}
        Even more, for any $n\in\mathbb{N},$ there is $x_n\in\Omega$ such that $M(t_n)=u(x_n,t_n)$
        and
        \[
            \lim_{n\to\infty}x_j=x_\infty\in B(u)
        \]
        where
        \[
            B(u)\coloneqq\left\{x\in\overline{\Omega}\colon
                \begin{aligned}
                &\exists (x_n,t_n) \in \Omega\times(0,T)
                \text{ s.t. } \\
                &(x_n,t_n)\to(x,T)
                 \text{ and }u(x_n,t_n)\to\infty
                \end{aligned}
            \right\}
        \]

        \medskip

        \noindent{\it Claim 1}: The sequences $\{(t_{n+1}-t_n)M^{p-1}(t_n)\}_{n\in\mathbb{N}}$ is bounded.

        \medskip

        Then there is a positive constant $C$ such that
        \[
            t_{n+1}-t_n\le CM^{1-p}(t_n)=C2^{n(1-p)}M^{1-p}(t_0).
        \]
        Therefore
        \[
            T-t_0\le CM^{1-p}(t_0)\sum_{n=0}^\infty 2^{n(1-p)}\le C\|u(\cdot,t_0)\|_{\lp[\infty](\Omega)}^{1-p},
        \]
        and we have the desired result.

        \medskip

        Then, we need to show the claim 1. Assume that the sequences $\{(t_{n+1}-t_n)M^{p-1}(t_n)\}_{n\in\mathbb{N}}$ is unbounded.
        Then, via a subsequence, still denote as $\{(t_{n+1}-t_n)M^{p-1}(t_n)\}_{n\in\mathbb{N}},$ we
        can assume that
        \begin{equation}
            \label{eq:mn}
            \lim_{n\to\infty} (t_{n+1}-t_n)M^{p-1}(t_n)=\infty.
        \end{equation}
        For any $n\in\mathbb{N},$ we set $M_n=M(t_n),$ $b_n=bM_n^{(1-p)(1-s)},$
        and
        \[
            \phi_n(y,\tau)=\frac{1}{M_n}u\left(M_n^{\frac{1-p}{2}}y+x_n, M_n^{1-p}
            \tau+t_n\right)
        \]
        for $(y,\tau)\in \Omega_n\times I_n$ where
        \[
            \Omega_n\coloneqq\{y\in\mathbb{R}^N\colon M_n^{\frac{1-p}{2}}y+x_n\in\Omega\}
            \text{ and }
            I_n\coloneqq(-t_nM_n^{p-1},(T-t_n)M_n^{p-1}).
        \]
        Then, for any $n\in \mathbb{N}$ we have
        \begin{equation}
            \label{eq:phin1}
            \phi_n(0,0)=1
         \end{equation}
         and
        \begin{equation}
            \label{eq:phin2}
              \phi_n(y,\tau)\le 2 \text{ in } \Omega_n\times (-t_n M_n^{p-1},(T-t_n)M_n^{p-1}).
         \end{equation}
        Furthermore,  $\phi_n$ is a solution of
        \begin{equation}
            \label{eq:phin3}
            \begin{cases}
                \phi_{n\tau}(y,\tau)-\om[a,b_n]\phi_n(y,\tau)=\phi_n^p(y,\tau) &\text{ in }\Omega_n\times I_n,\\
                \phi_n(y,\tau)=0&\text{ in }(\mathbb{R}^N\setminus\Omega_n)\times I_n,
            \end{cases}
        \end{equation}
        for all $n\in\mathbb{N}.$
        In additional, we have $I_n\to\mathbb{R}$ as $n\to \infty,$ and
        \[
            c<\text{dist}(x_n,\partial\Omega)=M_n^{1-p}\text{dist}(0,\Omega_n),
            \quad\forall n\in\mathbb{N}
        \]
        due to the blow up set of $u$ is in the interior of $\Omega.$ Therefore
        $\Omega_n\to\mathbb{R}^N$ as $n\to \infty.$ Thus, given $R>0,$
        there is $n_0=n(R)$ such that $B_R(0)\subset\Omega_n$ for any $n\ge n_0.$

           On the other han, following what was done in
            \cite{Regularidad}  and using \eqref{eq:phin2},
            we have that for any $n\ge n_0$
            there is $\delta=\delta(a,R)>0$ and $L=L(a,R)>0$ such that
            \[
                |\phi_n(z,0)-\phi_n(y,0)|\le L|z-y|^\delta\quad \forall z,y\in B_R(0).
            \]

        Thus, using \eqref{eq:phin1}, we get
        \[
            \phi_n(y,0)\ge g(y)\coloneqq (1-c|y|^\delta)_+\ge0 \quad \forall y\in B_R(0)
        \]
        for any $n\ge n_0.$
        Therefore, if $n\ge n_0,$ then $\phi_n$ is a super-solution of
        \begin{equation}
            \label{eq:auxh}
                \begin{cases}
                    h_{\tau}(y,s)-\om[a,b_n]h(y,\tau)=h^p(y,\tau)
                    	&\text{ in }B_R(0)\times J_n,\\
                    h(y,\tau)=0&\text{ in }(\mathbb{R}^N\setminus B_R(0))\times J_n,\\
                    h(y,0)=g(y)&\text{ in }B_R(0),
            \end{cases}
        \end{equation}
        where $J_n\coloneqq (0,(t_{n+1}-t_n)M_n^{p-1}).$

        \medskip

        \noindent{\it Claim 2:} For $R$ large enough there is $n_1\ge n_0$ such if $n\ge n_1$
            then the solution  of \eqref{eq:auxh}
           blows up at time $S_{R,n}<\infty.$ In additional, there is $S_R>0$ such that
           $S_{R,n}\le S_{R}$  for any $n\ge n_1.$

        \medskip

        This claim, the comparison principle and the fact that at
        $\phi_n$ is a super-solution of \eqref{eq:auxh} for all $n\ge n_1,$
        imply that $\phi_n$ blows up at time $S_n\le S_{R,n}\le S_R$ for all $n\ge n_1.$
        This is a contradiction with \eqref{eq:phin2}. Indeed, by \eqref{eq:mn}, there is
        $n_2>n_1$ such that $(t_{n+1}-t_n)M_n^{p-1}>S_R$ for all $n\ge n_2.$ Then
        $\{(t_{n+1}-t_n)M_n^{p-1}\}_{n\in\mathbb{N}}$ can not be unbounded and therefore
        the claim 1 holds.

        \medskip

        To conclude the proof, we need to show the claim 2.

        \medskip

        We can now argue almost exactly as in the proof of
        Theorem \ref{cpconvasym}, to take $R>0$ such that
        \begin{equation}
            \label{eq:sobreR}
            \mathcal{J}\coloneqq\int_{B_R(0)}
            g(y)\phi_R(y) dy >\sigma_1(B_R(0))^\frac{1}{p-1}
        \end{equation}
        where $\sigma_1(B_R(0))$ denotes the first Dirichlet eigenvalue of
        $-\Delta$ in $B_R(0)$ and $\phi_R$ is the positive eigenfunction
        of $-\Delta$ associated to $\sigma_1(B_R(0))$ such that
        $\|\phi_R\|_{\lp[1](B_R(0))}=1.$

        \medskip

        We now apply the Kaplan's argument to prove the claim 2. For any
        $n\ge n_0(R),$ we take
        $h_n$  the solution of \eqref{eq:auxh} and
        \[
            J_n(t)=\int_{B_R(0)}h_n(y,t)\psi_n(y) dy,
        \]
        where $\psi_n$ is the positive eigenfunction
        of $-\om[a,b_n]$ associated to the first eigenvalue
        $\lambda_1^{a,b_n}(B_R(0))$ such that
        $\|\psi_n\|_{\lp[1](B_R(0))}=1.$ Then
        \[
            J_{nt}=-\lambda_1^{a,b_n}(B_R(0)) J_n + \int_{B_R(0)} h_n(y)\psi_{n}(y) dy
                        \ge J_n^p-\lambda_1^{a,b_n}(B_R(0)) J_n.
        \]
        By Lemma \ref{lema:ba0} and \eqref{eq:sobreR}, there is $n_1>n_0$ such that
        \[
           J_n(0) =\int_{B_R(0)} g(y)\psi_n(y) dy>\lambda^{a,b_n}_1(B_R(0))^{\frac{1}{p-1}}
        \]
        for any $n\ge n_1.$ Then, $h_n$ blows up at time
        \[
            S_{R,n}  \le K_{R,n}\coloneqq-\frac{1}{(p-1)\lambda_1^{a,b_n}(B_R(0))}
                \ln\left(1-\lambda_1^{a,b_n}(B_R(0))J_n(0)^{1-p}\right)
        \]
        for any $n\ge n_1.$ By Lemma \ref{lema:ba0}, we get
        \[
           \lim_{n\to\infty} K_{R,n}=
           -\frac{1}{(p-1)\sigma_1(B_R(0))}
                \ln\left(1-\sigma_1(B_R(0))\mathcal{J}^{1-p}\right).
        \]
        Therefore, there is $S_R$ such that $S_{R,n}\le S_R$ for any $n\ge n_1.$
    \end{proof}

\section*{Acknowledgements}

    L.M.D.P. was partially
    supported by the European Union’s Horizon 2020 research and innovation program
    under the Marie Sklodowska-Curie grant agreement No 777822,
    (Agencia Nacional de Promoci\'on de la Investigaci\'on, el Desarrollo
    Tecnol\'ogico y la Innovaci\'on PICT-2018-3183, and
    PICT-2019-00985 and UBACYT 20020190100367.

R. F. was partially supported by the Spanish project  PID2023-146931NB-I00 and Grupo de Investigaci\'on UCM 920894.

\end{document}